\documentclass[preprint, 12pt, english]{elsarticle}
\usepackage{amsmath}
\usepackage{latexsym, amssymb}
\usepackage{amsthm}

\newtheorem{thm}{Theorem}[section] 
\newtheorem{algo}[thm]{Algorithm}

\newtheorem{prop}[thm]{Proposition}

\newcommand\operA[2]{{\if!#2!\operatorname{#1}\else{\operatorname{#1}_{#2}^{\phantom{I}}}\fi}} 
\newcommand\set[1]{\{#1\}}

\newcommand\Cref[1]{{Corollary~\ref{#1}}}%

\def\tr{{\operatorname{Tr}}}

\def\norm{{\operatorname{N}}}

\newcommand{\Trace}[1][]{\if!#1!\operatorname{Tr}\else{\operatorname{Tr}_{#1}^{\phantom{I}}}\fi} 

\long\def\forget#1\forgotten{{}} %

\def\({\left(}
\def\){\right)}

\newif\iffurther
\furtherfalse

\newif\ifXY 
\XYtrue     
\ifXY

\input xy
\input xyidioms.tex
\usepackage{xy}
\xyoption{all} %
\fi 

\usepackage{babel}

\journal{??}

\begin{document}

\begin{frontmatter}

\title{Quaternion quadratic equations in characteristic 2}

\author{Adam Chapman\corref{cor2}}
\ead{adam1chapman@yahoo.com}
\cortext[cor2]{The author is supported by Wallonie-Bruxelles International.}
\address{ICTEAM Institute, Universit\'{e} Catholique de Louvain, B-1348 Louvain-la-Neuve, Belgium}

\begin{abstract}
In this paper we present a solution for any standard quaternion quadratic equation, i.e. an equation of the form $z^2+\mu z+\nu=0$ where $\mu$ and $\nu$ belong to some quaternion division algebra $Q$ over some field $F$, assuming the characteristic of $F$ is $2$.
\end{abstract}

\begin{keyword}
Quaternion algebra, field of characteristic 2, quadratic equation
\MSC[2010] primary 16K20; secondary 11R52
\end{keyword}

\end{frontmatter}

\section{Introduction}
In \cite{HuangSo}, Huang and So presented a solution for any quadratic equation $z^2+\mu z+\nu=0$ over Hamilton's quaternion algebra $\mathbb{H}$.
In \cite{Abrate}, Abrate generalized that result for any quaternion algebra over any field of characteristic not $2$.
Being able to solve a quaternion quadratic equation has proved useful, for example in computing the left eigenvalues of a $2 \times 2$ quaternion matrix (see \cite{Wood}).
In this paper we shall present a solution for such equations over a quaternion division algebra over a field of characteristic $2$.

Let $F$ be a field of characteristic 2.
A quaternion algebra $Q$ over $F$ is a four dimensional algebra $F+F x+F y+F x y$ where $x$ and $y$ satisfy the relations
$$x^2+x=\alpha, y^2=\beta, x y+y x=y$$
for some $\alpha \in F$ and $\beta \in F^\times$.
Every central simple algebra over $F$ of dimension $4$ (or equivalently of degree $2$) is a quaternion algebra \cite[Chapter 8, Section 11]{Scharlau}.
The quaternion algebra is equipped with a canonical involution $\sigma$ defined by
$$\sigma(a+b x+c y+d x y)=a+b+b x+c y+d x y$$
for any $a,b,c,d \in F$.
For any element $q \in Q$, $\sigma(q)$ is called its ``conjugate".
The norm and trace of $q$ are defined to be $\norm(q)=q \sigma(q)$ and $\tr(q)=q+\sigma(q)$, both are in $F$.
These definitions coincide with the general definitions of the reduced norm and trace in central simple algebras.
For any $q \in Q$, $q^2+\tr(q) q+\norm(q)=0$, which means that $\tr(q)=0$ if and only if $q^2 \in F$.
The space $F+F y+F x y$ consists of all the elements of trace zero.

It is known that $Q$ is either a division algebra or the matrix algebra $M_2(F)$.
From now on we shall assume that $Q$ is a division algebra.
In particular it means that $F$ must be an infinite field, following \cite{Maclagan-Wedderburn}.
From \cite{Herstein} it is known that the quadratic equation $z^2+\mu z+\nu=0$ has either infinitely many roots or up to two roots.

The elements $z \in Q \setminus F$ which satisfy $z^2+z \in F$ are called ``Artin-Schreier", and the elements $z \in Q \setminus F$ which satisfy $z^2 \in F$ are called ``square-central".
In particular, in the description of the quaternion algebra above, $x$ is Artin-Schreier and $y$ is square-central.
It is pointed out in \cite[Chapter 8, Section 11]{Scharlau} that for any Artin-Schreier $x' \in Q$ there exists a square-central element $y' \in Q$ satisfying $x' y'+y' x'=y'$, and then $x'$ and $y'$ can replace $x$ and $y$ in the description of the quaternion algebra above. The canonical involution, norm and trace are independent of the choice of generators and therefore remain the same.
Given this Artin-Schreier element $x'$, $Q=V_0+V_1$ where $V_0=F+F x'$ and $V_1=F y'+F x' y'$. Every element in $V_0$ commutes with $x'$ and every element $t \in V_1$ satisfies $t x+x t=t$.

\begin{prop}
For any square-central element $y' \in Q$ there exists an Artin-Schreier element $x' \in Q$ satisfying $x' y'+y' x'=y'$.
\end{prop}

\begin{proof}
By a straight-forward computation, for any $z \in Q$, $z y'+y' z$ commutes with $y'$.
If $z y'+y' z=0$ for every $z \in Q$ then $y'$ is central in $Q$, which means that $y' \in F$, contradictory to the assumption that $y'$ is square-central.
Therefore we can choose some $z$ for which $w=z y'+y' z \neq 0$.
Since $Q$ is a division algebra, $w$ is invertible.
Set $x'=y' w^{-1} z$. Then $x' y'+y' x'=y'$.
By a straight-forward computation, $x'^2+x'$ commutes with $y'$.
Since $Q$ is a division algebra of degree $2$ \cite{Albert}, it has no nontrivial division subalgebras, which means that the subalgebra generated by $x'$ and $y'$ is the whole algebra. Since $x'^2+x'$ commutes with $x'$ and $y'$, $x'^2+x' \in F$.
\end{proof}

\section{Quaternion Quadratic Equations}
Let $F$ be a field of characteristic 2 and $Q$ be a quaternion division algebra over $F$.
Let $z^2+\mu z+\nu=0$ be the equation under discussion.
The coefficients $\mu$ and $\nu$ are arbitrary elements in $Q$.

As a set, $Q$ can be written as the disjoint union of $\set{0}$, $F^\times$, the set of square-central elements, and all the other elements.
We shall denote the latter by $Q'$.
Therefore, $\mu$ belongs to exactly one of these subsets.
If $\mu \in F^\times$ then by dividing the equation by $\mu^2$ we get $(\mu^{-1} z)^2+(\mu^{-1} z)+\mu^{-2} \nu=0$, which means that in this situation it suffices to be able to solve the case of $\mu=1$.
If $\mu \in Q'$ then it satisfies $\mu^2+\eta \mu \in F$ for some $\eta \in F^\times$.
Consequently $(\eta^{-1} \mu)^2+(\eta^{-1} \mu) \in F$, which means that $\eta^{-1} \mu$ is Artin-Schreier and the equation can be expressed as $(\eta^{-1} z)^2+(\eta^{-1} \mu) (\eta^{-1} z)+\eta^{-2} \nu=0$, and so in this situation it suffices to be able to solve the case of an Artin-Schreier $\mu$.

In conclusion, it suffices to be able to solve the following cases:
\begin{enumerate}
\item $\mu$ is Artin-Schreier.
\item $\mu$ is square-central.
\item $\mu=1$.
\item $\mu=0$.
\end{enumerate}

\section{$\mu$ is Artin-Schreier}

Assume that the equation is $z^2+\mu z+\nu=0$ for some $\nu \in Q$ and some Artin-Schreier element $\mu$ satisfying $\mu^2+\mu=\alpha$.
The element $\nu$ splits as $\nu_0+\nu_1$ where $\nu_0 \mu=\mu \nu_0$ and $\nu_1 \mu+\mu \nu_1=\nu_1$.
Furthermore, $\nu_0=\nu_{0,0}+\nu_{0,1} \mu$ where $\nu_{0,0},\nu_{0,1} \in F$.

\begin{thm}
If $\nu_1 \neq 0$ then all the elements $z \in Q$ satisfying $z^2+\mu z+\nu=0$ belong to the set
\begin{eqnarray*}
R & = & \{ b^2+b+\nu_{0,1}+b \mu+(b+\mu)^{-1} \nu_1 :  b \in F,\\
 & & (b^2+b+\nu_{0,1})^2+(b^2+b) \alpha+(b^2+b+\alpha)^{-1} \nu_1+\nu_{0,0}=0 \}.
\end{eqnarray*}
Otherwise, all the elements $z \in Q$ belong to the subfield $F[\mu]$, and therefore we can simply solve it as a quadratic equation over this field.
\end{thm}

\begin{proof}
Since $\mu$ is Artin-Schreier, $z=z_0+z_1$ where $z_0 \mu=\mu z_0$ and $z_1 \mu+\mu z_1=z_1$.

Furthermore, the expression $I=z^2+\mu z+\nu$ splits into two parts $I_0+I_1$ such that $I_0 \mu=\mu I_0$ and $I_1 \mu+\mu I_1=I_1$.

The equation then splits in the following way:
\begin{eqnarray*}
I_0 & = & z_0^2+z_1^2+\mu z_0+\nu_0=0\\
I_1 & = & z_0 z_1+z_1 z_0+\mu z_1+\nu_1=0
\end{eqnarray*}

Now, $z_0=a+b \mu$ for some $a,b \in F$. Consequently $z_0 z_1+z_1 z_0=b z_1$. Moreover, $z_0^2=a^2+b^2 \mu+b^2 \alpha$ and $\mu z_0=a \mu+b \mu+b \alpha$
Similarly $\nu_0=\nu_{0,0}+\nu_{0,1} \mu$ for some $\nu_{0,0},\nu_{0,1} \in F$.

The second part of the equations therefore becomes
$b z_1+\mu z_1+\nu_1=0$, which means that $(b+\mu) z_1=\nu_1$.
Now, $((b+\mu) z_1)^2=(b+\mu) z_1 (b+\mu) z_1=(b+\mu) (b+\mu+1) z_1^2=(b^2+b+\alpha) z_1^2$.
If $\nu_1 \neq 0$ then $z_1^2 \neq 0$ and $b^2+b+\alpha \neq 0$, and we obtain $z_1^2=(b^2+b+\alpha)^{-1} \nu_1$.

If $\nu_1=0$ then either $z_1 = 0$ or $b^2+b+\alpha=0$. Since $Q$ is a division algebra, there is no $b \in F$ for which $b^2+b+\alpha=0$, and therefore $z_1=0$, and all the elements $z \in Q$ satisfying $z^2+\mu z+\nu=0$ belong to $F[\mu]$, which means that the equation can be simply solved in the subfield $F[\mu]$.

Assume $\nu_1 \neq 0$.

The first part splits again $I_0=I_{0,0}+I_{0,1}$ where $I_{0,0} \in F$ and $I_{0,1} \in F \mu$.
It splits in the following way:
\begin{eqnarray*}
I_{0,0} & = & a^2+b^2 \alpha+(b^2+b+\alpha)^{-1} \nu_1+b \alpha+\nu_{0,0}=0\\
I_{0,1} & = & b^2 \mu+(a+b) \mu+\nu_{0,1} \mu=0
\end{eqnarray*}

From $I_{0,1}$ we obtain $a=b^2+b+\nu_{0,1}$.
By substituting that in $I_{0,0}$ we obtain $(b^2+b+\nu_{0,1})^2+(b^2+b) \alpha+(b^2+b+\alpha)^{-1} \nu_1+\nu_{0,0}=0$.
\end{proof}

As a result we have the following algorithm for calculating the roots of the equation $z^2+\mu z+\nu=0$ where $\mu$ is Artin-Schreier and $\nu \not \in F[\mu]$:

\begin{algo}
\begin{enumerate}
\item Calculate all the elements $t \in F$ satisfying $(t+\nu_{0,1})^2+t \alpha+(t+\alpha)^{-1} \nu_1+\nu_{0,0}=0$.
\item For each such $t$, find all the elements $b \in F$ satisfying $b^2+b=t$.
\item For each such $b$ (there should be up to $6$ of those in total), substitute the element $b^2+b+\nu_{0,1}+b \mu+(b+\mu)^{-1} \nu_1$ in the original equation to check whether it is really a root. The set of roots consists of all the elements who passed the substitution test.
\end{enumerate}
\end{algo}

\section{$\mu$ is square-central}

Assume $\mu$ is square-central, then there exists some Artin-Schreier $\theta$ satisfying $\theta \mu+\mu \theta=\mu$.
In particular, $\mu^2=\beta$ and $\theta^2+\theta=\alpha$ for some $\alpha,\beta \in F$.

The element $\nu$ splits into $\nu_0+\nu_1$ where $\nu_0 \theta=\theta \nu_0$ and $\nu_1 \theta+\theta \nu_1=\nu_1$.
Furthermore, $\nu_0=\nu_{0,0}+\nu_{0,1} \theta$ and $\nu_1=\nu_{1,0} \mu+\nu_{1,1} \mu \theta$.

\begin{thm}
All the elements $z \in Q$ satisfying $z^2+\mu z+\nu=0$ belong to
\begin{eqnarray*}
R & = & \{b c+\nu_{1,0}+b \theta+c \mu+\beta^{-1} (b^2+\nu_{0,1}) \mu \theta : b,c \in F,\\
 & & \beta^{-1} (b^2+\nu_{0,1}) b+b+\nu_{1,1}=0,\\
 & & (b c+\nu_{1,0})^2+b^2 \alpha+\beta (c^2+c d+\alpha d^2)+c \beta+\nu_{0,0}=0\}.
\end{eqnarray*}
\end{thm}

\begin{proof}
The element $z$ splits into $z_0+z_1$ where $z_0 \theta=\theta z_0$ and $z_1 \theta+\theta z_1=z_1$.

The expression $I=z^2+\mu z+\nu$ splits into $I_0+I_1$ where $I_0 \theta=\theta I_0$ and $I_1 \theta+\theta I_1=I_1$.

The equation correspondingly splits into two:
\begin{eqnarray*}
I_0 & = & z_0^2+z_1^2+\mu z_1+\nu_0=0\\
I_1 & = & z_0 z_1+z_1 z_0+\mu z_0+\nu_1=0
\end{eqnarray*}

Now, $z_0=a+b \theta$ and $z_1=c \mu+d \mu \theta$ for some $a,b,c,d \in F$.

$z_0 z_1+z_1 z_0=b z_1=b c \mu+b d \mu \theta$.

$\mu z_0=a \mu+b \mu \theta$, $\mu z_1=c \beta+d \beta \theta$.

$z_0^2=a^2+b^2 \theta+b^2 \alpha$, $z_1^2=\beta (c^2+c d+\alpha d^2)$.

Each of the parts $I_0$ and $I_1$ splits again into two parts $I_0=I_{0,0}+I_{0,1} \theta$ and $I_1=I_{1,0} \mu+I_{1,1} \mu \theta$ where $I_{0,0},I_{0,1},I_{1,0},I_{1,1} \in F$.

Consequently we have the following system of four equations:

\begin{eqnarray*}
I_{0,0} & = & a^2+b^2 \alpha+\beta (c^2+c d+\alpha d^2)+c \beta+\nu_{0,0}=0\\
I_{0,1} & = & b^2+d \beta+\nu_{0,1}=0\\
I_{1,0} & = & b c+a+\nu_{1,0}=0\\
I_{1,1} & = & b d+b+\nu_{1,1}=0
\end{eqnarray*}

From $I_{0,1}$ we obtain $d=\beta^{-1} (b^2+\nu_{0,1})$.
Substituting that in $I_{1,1}$, we obtain $\beta^{-1} (b^2+\nu_{0,1}) b+b+\nu_{1,1}=0$.

From $I_{1,0}$ we obtain $a=b c+\nu_{1,0}$, and by substituting that in equation $I_{0,0}$ we obtain $(b c+\nu_{1,0})^2+b^2 \alpha+\beta (c^2+c d+\alpha d^2)+c \beta+\nu_{0,0}=0$.
\end{proof}

As a result we have the following algorithm for calculating the roots of the equation $z^2+\mu z+\nu=0$ where $\mu$ is square-central:

\begin{algo}
\begin{enumerate}
\item Calculate all the elements $b \in F$ satisfying $\beta^{-1} (b^2+\nu_{0,1}) b+b+\nu_{1,1}=0$.
\item For each such $b$, find all the elements $c \in F$ satisfying $(b c+\nu_{1,0})^2+b^2 \alpha+\beta (c^2+c d+\alpha d^2)+c \beta+\nu_{0,0}=0$.
\item For each such pair $b$ and $c$ (there should be up to $6$ of those in total), substitute the element $b c+\nu_{1,0}+b \theta+c \mu+\beta^{-1} (b^2+\nu_{0,1}) \mu \theta$ in the original equation to check whether it is really a root. The set of roots consists of all the elements who passed the substitution test.
\end{enumerate}
\end{algo}

\section{$\mu=1$}
Every element in $\nu \in Q$ is one of the following: central, square-central or $n x$ for some Artin-Schreier element $x$ and $n \in F^\times$.

\begin{thm}
\begin{enumerate}
\item If $\nu \in F$ then all the elements $z \in Q$ satisfying $z^2+z+\nu=0$ belong to $F$.
\item If $\nu$ is square-central then the elements $z \in Q$ satisfying $z^2+z+\nu=0$ are all the elements of the form $a+\nu$ where $a$ satisfies $a^2+a+\nu^2=0$.
\item If $\nu=n x$ for some Artin-Schreier element $x$ and $n \in F^\times$ then all the elements $z \in Q$ satisfying $z^2+z+\nu=0$ belong to $F[\nu]$.
\end{enumerate}
\end{thm}

\begin{proof}

\textbf{Case 1}

Assume $\nu \in F$. Fix some Artin-Schreier element $x$.

Every element $z \in Q$ decomposes as $z_0+z_1$ where $z_0 x=x z_0$ and $z_1 x+x z_1=z_1$.
The expression $I=z^2+z+\nu$ decomposes similarly into $I_0+I_1$, and we obtain the following system of equations:

\begin{eqnarray*}
I_0 & = & z_0^2+z_1^2+z_0+\nu=0\\
I_1 & = & z_0 z_1+z_1 z_0+z_1=0
\end{eqnarray*}

The element $z_0$ is equal to $a+b x$ for some $a,b \in F$.
Substituting that in $I_1$ leaves $(b+1) z_1=0$.
If $z_1 \neq 0$ then $b=1$. Then $I_0=a^2+\alpha+1+z_1^2+a+x+\nu=0$.
However, $I_0=I_{0,0}+I_{0,1} x$ where $I_{0,0},I_{0,1} \in F$. In particular, $I_{0,1}=1$, which means that $1=0$ and that creates a contradiction.
Therefore $z$ commutes with $x$, which means that $z \in F[x]$. However, this is true for every Artin-Schreier element $x$, and therefore $z \in F$.

\bigskip
\textbf{Case 2}

Assume $\nu$ is square-central. Then for some fixed Artin-Schreier element $x$ satisfying $x^2+x=\alpha$ we have $x \nu+\nu x=\nu$.

Every element $z \in Q$ decomposes as $z_0+z_1$ where $z_0 x=x z_0$ and $z_1 x+x z_1=z_1$.
The expression $I=z^2+z+\nu$ decomposes similarly into $I_0+I_1$, and we obtain the following system of equations:
\begin{eqnarray*}
I_0 & = & z_0^2+z_1^2+z_0=0\\
I_1 & = & z_0 z_1+z_1 z_0+z_1+\nu=0
\end{eqnarray*}

$I_1=(b+1) z_1+\nu=0$, which means $\nu=(b+1) z_1$.
Since $\nu \neq 0$, $b+1$ is invertible, and $z_1=(b+1)^{-1} \nu$.

$I_0=a^2+b^2 \alpha+b^2 x+(b+1)^{-1} \nu^2+a+b x=0$

$I_0$ splits into $I_{0,0}+I_{0,1} x$, as follows
\begin{eqnarray*}
I_{0,0} & = & a^2+b^2 \alpha+(b+1)^{-1} \nu^2+a=0\\
I_{0,1} & = & b^2+b=0
\end{eqnarray*}

From $I_{0,1}$ we obtain either $b=0$ or $b=1$.
The second option is not possible, because $b+1$ is invertible.
Consequently $b=0$.

Then $I_{0,0}=a^2+a+\nu^2=0$.

In conclusion, the roots are the elements of the form $a+\nu$ where $a$ satisfies $a^2+a+\nu^2=0$.
\end{proof}

\bigskip
\textbf{Case 3}

Assume $\nu=n x$ for some Artin-Schreier $x$ and some $n \in F^\times$.
Every element $z \in Q$ decomposes as $z_0+z_1$ where $z_0 x=x z_0$ and $z_1 x+x z_1=z_1$.
The expression $I=z^2+z+\nu$ decomposes similarly into $I_0+I_1$, and we obtain the following system of equations:

\begin{eqnarray*}
I_0 & = & z_0^2+z_1^2+z_0+\nu=0\\
I_1 & = & z_0 z_1+z_1 z_0+z_1=0
\end{eqnarray*}

The element $z_0$ decomposes as $a+b x$ for some $a,b \in F$, and so from $I_1$ we obtain $(b+1) z_1=0$, which means that $z_1=0$ or $b=1$.
If $z_1 \neq 0$ then $b=1$.
Consequently $z_0^2=a^2+x+\alpha$ and $z_0=a+x$.

Now, $I_0=a^2+x+\alpha+z_1^2+a+x+\nu=a^2+\alpha+z_1^2+a+\nu=0$

However $I_0=I_{0,0}+I_{0,1} x$ where $I_{0,0},I_{0,1} \in F$. In particular, $I_{0,1}=\nu=0$, which creates a contradiction.
Consequently, it is enough to solve the equation over the field $F[\nu]$.

\section{$\mu=0$}

If $\nu \in F$ then the equation has a root only if $\nu=a^2 \alpha+(a b+b^2) \beta+c^2$ for some $a,b,c,\alpha \in F$ and $\beta \in F^\times$ where $Q=F[x,y : x^2+x=\alpha, y^2=\beta, x y+y x=y]$, because these are the squares of all the square-central and central elements in $F$.
If there are such $a,b,c$ then $z=a x y+b y+c$ is a root. Similarly, all its conjugates are roots as well, and there are infinitely many of them.

\begin{thm}
\begin{enumerate}
\item If $\nu$ is a square-central element then all the elements $z \in Q$ satisfying $z^2+\nu=0$ belong to $F[x]$ where $x$ satisfies $x \nu+\nu x=\nu$.
\item If $\nu=n x$ for some Artin-Schreier element $x$ and $n \in F^\times$ then all the elements $z \in Q$ satisfying $z^2+\nu=0$ belong to $F[\nu]$.
\end{enumerate}
\end{thm}

\begin{proof}
\textbf{Case 1}

Assume $\nu$ is square-central.
Then $x \nu+\nu x=\nu$ for some Artin-Schreier $x$.
As before, $z=z_0+z_1$ and we obtain the system
\begin{eqnarray*}
I_0 & = & z_0^2+z_1^2=0\\
I_1 & = & z_0 z_1+z_1 z_0+\nu=0
\end{eqnarray*}

From $I_1$ we obtain $b z_1=\nu$ which means that either $z_1=0$ or $b=0$.
If $z_1 \neq 0$ then $b=0$ and so $I_0=a^2+z_1^2=0$. But this means that $z_1+a$ is a zero divisor, and that creates a contradiction.
Consequently $z_1=0$, which means that all the roots can be obtained by solving the equation over the field $F[x]$.

\bigskip
\textbf{Case 2}

Assume $\nu=n x$ for some Artin-Schreier element $x$ and some $n \in F^\times$.
As before $z=z_0+z_1$, and we obtain the system
\begin{eqnarray*}
I_0 & = & z_0^2+z_1^2+\nu=0\\
I_1 & = & z_0 z_1+z_1 z_0=0
\end{eqnarray*}

$I_1=b z_1=0$ which means that either $b=0$ or $z_1=0$.
If $z_1 \neq 0$ then $b=0$ and so $I_0=a^2+z_1^2+\nu=0$.
However, $I_0=I_{0,0}+I_{0,1} x$ and $I_{0,1}=n x=0$, which creates a contradiction.
Again, this equation can be solved simply over the field $F[x]=F[\nu]$.
\end{proof}

\section*{Acknowledgements}
I owe thanks to Jean-Pierre Tignol and Uzi Vishne for their help and support.

\section*{Bibliography}
\bibliographystyle{amsalpha}
\bibliography{bibfile}
\end{document}